\documentclass[12pt]{article}
\usepackage[bookmarks=true,
colorlinks=true,
linkcolor=black]{hyperref}
\usepackage{amsmath}
\usepackage{amsthm}
\usepackage{fancyhdr}
\usepackage{bbm} 
\usepackage{color}
\usepackage{amssymb}
\usepackage{subfig}
\usepackage{sidecap}
\usepackage{graphicx}
\usepackage{geometry}
\usepackage{lineno}
\usepackage[all]{xy}
\usepackage[mathscr]{eucal}

\theoremstyle{plain}
\newtheorem{theorem}{Theorem}[section]
\newtheorem{lemma}[theorem]{Lemma}
\newtheorem{definition}[theorem]{Definition}
\newtheorem{corollary}[theorem]{Corollary} 
\newtheorem{proposition}[theorem] {Proposition}

\newcommand{\real}{\ensuremath{\mathbb{R}}}
\newcommand{\cone}{\ensuremath{\mathbb{C}}}
\newcommand{\homol}{\ensuremath{\mathbb{H}}}

\newcommand{\lp}{\ensuremath{\mathbb{L}_p}}
\newcommand{\mm}[2]{\ensuremath{\mathbb{#1}_{#2}}}
\newcommand{\borel}{\ensuremath{\mathcal{P}}}
\newcommand{\diagram}{\ensuremath{\mathcal{D}}}
\newcommand{\chemin}{\ensuremath{\mathscr{C}}}
\newcommand{\chemind}{\ensuremath{\mathscr{C}_D}}
\newcommand{\gw}{\ensuremath{\mathscr{M}}}
\newcommand{\fgh}{\ensuremath{\mathscr{F}}}
\newcommand{\ecc}[1]{\ensuremath{\Theta_{#1}}}

\newcommand{\prer}{\ensuremath{d_{HT}}}
\newcommand{\var}{\ensuremath{\hat{d}_{HT}}}

\begin{document}

\renewcommand\thelinenumber{\color[rgb]{0.2,0.5,0.8}\normalfont\sffamily\scriptsize\arabic{linenumber}\color[rgb]{0,0,0}}
\renewcommand\makeLineNumber {\hss\thelinenumber\ \hspace{6mm} \rlap{\hskip\textwidth\ \hspace{6.5mm}\thelinenumber}}

%\linenumbers

\title{A Topological Study of Functional Data and Fr\'{e}chet Functions of Metric Measure
Spaces\thanks{Research supported in part by NSF grants DMS-1722995 and DMS-1723003.}}
\author{Haibin Hang$^\dag$, Facundo M\'{e}moli$^\ddag$, Washington Mio$^\dag$\\
\small $^\dag$Department of Mathematics, Florida State University\\ 
\small Tallahassee, FL 32306-4510 USA\\
\small hhang@math.fsu.edu \qquad wmio@fsu.edu\\
\small $^\ddag$Department of Mathematics and Department of Computer Science and Engineering\\
\small Ohio State University, Columbus, OH 43210-1174 USA\\
\small memoli@math.osu.edu}

\date{}

\maketitle
%%%%%%%%%%%%%%%%%%%%%%%%%%%%%%%%%%%%%%%%%%%%%%%%%%%%%%%%%%
%%%%%%%%%%%%%%%%%%%%%%%%%%%%%%%%%%%%%%%%%%%%%%%%%%%%%%%%%%
%%%%%%%%%%%%%%%%%%%%%%%%%%%%%%%%%%%%%%%%%%%%%%%%%%%%%%%%%%


\abstract
{We study the persistent homology of both functional data on compact topological spaces
and structural data presented as compact metric measure spaces. One of our goals is to
define persistent homology so as to capture primarily properties of the shape
of a signal, eliminating otherwise highly persistent homology classes that may exist simply because
of the nature of the domain on which the signal is defined.  We investigate the stability of these
invariants using metrics that downplay regions where signals are weak. The
distance between two signals is small if they exhibit high similarity in regions
where they are strong, regardless of the nature of their full domains, in particular allowing
different homotopy types. Consistency  and estimation of persistent homology of metric measure
spaces from data are studied within this framework. We also apply the methodology to the
construction of multi-scale topological descriptors for data on compact Riemannian manifolds via
metric relaxations derived from the heat kernel.
}

%%%%%%%%%%%%%%%%%%%%%%%%%%%%%%%%%%%%%%%%%%%%%%%%%%%%%%%%%%
%%%%%%%%%%%%%%%%%%%%%%%%%%%%%%%%%%%%%%%%%%%%%%%%%%%%%%%%%%
%%%%%%%%%%%%%%%%%%%%%%%%%%%%%%%%%%%%%%%%%%%%%%%%%%%%%%%%%%
\section{Introduction}

This paper investigates ways of producing robust, informative summaries of both 
functional data on topological spaces and structural data in metric spaces, two
problems that permeate the sciences and applications. Many problems involving 
structural data may be formulated in the realm of metric-measure spaces
($mm$-spaces) $\mm{X}{\alpha} = (X, d_X, \alpha)$, where $(X,d_X)$ is
a metric space and $\alpha$ is a Borel
probability measure on $X$. For example, a dataset $x_1, \ldots, x_n \in X$
may be analyzed via the associated empirical measure
$\sum_{i=1}^n \delta_{x_i} /n$, where $\delta_{x_i}$ is the Dirac measure
based at $x_i$. 

The mean of a Euclidean distribution is a most basic statistic that
may be generalized to triples $(X,d_X,\alpha)$ via the Fr\'{e}chet function
$V_{\mm{X}{\alpha}} \colon X \to \real$ defined by
\begin{equation}
V_{\mm{X}{\alpha}} (x) = \int_X d^2(x, y) \, d \alpha (y)\,.
\label{E:frechet}
\end{equation}
In the Euclidean case, it is well known that the mean is the unique
minimizer of $V_{\mm{X}{\alpha}}$ , provided that $\alpha$ has finite second moment.
In the more general setting, a {\em Fr\'{e}chet mean} is a minimizer of $V_{\mm{X}{\alpha}} $,
not necessarily unique, thus leading to the concept of Fr\'{e}chet mean set.
One also may consider the local minima of $V_{\mm{X}{\alpha}}$ that often yield valuable
additional information about the distribution. The Fr\'{e}chet mean has
received a great deal of attention from  many authors (cf.\,\cite{gk73,patbat03,arn-bar13}).  
However, Fr\'{e}chet mean sets may be difficult to estimate or compare
quantitatively, limiting its applicability in data analysis.

Beyond means, a wealth of structural information about $\mm{X}{\alpha}$
typically resides in $V_{\mm{X}{\alpha}} $ (cf.\,\cite{diazetal18,diazetal18a}). One of our
primary goals is to carry out a topological study of the Fr\'{e}chet function via persistent
homology to uncover and summarize properties
of the shape of probability distributions on metric spaces. The study is
done in the general setting of Fr\'{e}chet functions of order
$p\geq1$, defined as
\begin{equation}
V_{p,\mm{X}{\alpha}} (x) = \int_X d^p(y,x) \, d \alpha (y)\,.
\label{E:frechetp}
\end{equation}
\noindent
Some authors refer to $V_{p, \mm{X}{\alpha}}$ as the $p$-eccentricity
function of $\mm{X}{\alpha}$ (cf.\,\cite{carlsson09}). 
Closely related to $V_{p, \mm{X}{\alpha}}$ is the function
$\sigma_{p, \mm{X}{\alpha}} \colon X \to \real$ defined as
\begin{equation}
\sigma_{p, \mm{X}{\alpha}} (x) = (V_{p, \mm{X}{\alpha}} (x))^{1/p} \,,
\label{E:ecc}
\end{equation}
which we refer to as the {\em $p$-centrality function} of $\mm{X}{\alpha}$.
We opt to construct topological summaries for $\mm{X}{\alpha}$
using $\sigma_{p, \mm{X}{\alpha}}$, rather than $V_{p, \mm{X}{\alpha}}$, because barcodes
or persistence diagrams derived from $p$-centrality are
more amenable to analysis. A key property of Fr\'{e}chet and centrality
functions is that they attenuate the influence of outliers in persistence
homology computed from large random samples of a distribution, a fact
that is implicit in the stability and consistency results proven in this paper.

We view $\sigma_{p, \mm{X}{\alpha}} \colon X \to \real$ as a
``signal'' on $X$ and first work in the more general setting of functional data
on compact topological spaces. Subsequently, we specialize to centrality
functions of $mm$-spaces. In the functional setting, a data object is a triple
$\mm{X}{f} :=(X, \tau_X, f)$, where $(X, \tau_X)$ is a topological space and
$f \colon X \to \real$ is a continuous function. The collection of all such triples
is denoted $\fgh$. We make the convention that large values of $f$ correspond
to weak signals. Generally, this is more consistent with the behavior of Fr\'{e}chet
functions that attain larger values in data deserts, far away from
concentrations of probability mass (cf.\,\cite{diazetal18}). 

In defining persistent homology \cite{frosini92, robins99,edelsbrunner-etal,
carlsson09,edehar08,cohenetal07} invariants for $\mm{X}{f}$, there
are a few drawbacks in following the standard procedure of using the filtration of $X$
given by the sublevel sets of $f$ directly. For example, a region of non-trivial topology
where the signal $f$ is weak might contribute highly persistent homology classes,
masking the ``real'' topology of the signal and producing a confounding effect.
A closely related problem is that, once we reach the maximum value of $f$,
the sublevel sets of $f$ coincide with the full space $X$, so that the global homology of
$X$ is the dominant information captured in a homological barcode
regardless of the ``support'' of the signal. The number of bars of infinite length for
$i$-dimensional homology is the $i$th Betti number of $X$.
At the barcode level, we could alleviate the problem by trimming
barcodes at the maximum value of $f$. However, stability results
for truncated barcodes obtained by factoring through the stability of persistent homology
of the usual sublevel set filtration would be somewhat weak because it would
require similarity of the barcodes prior to truncation. 
Thus, we introduce a metric on the space of functional topological spaces
with respect to which small distances indicate similarity where signals are strong. We
investigate the stability of persistent homology in this setting. We take the
homotopy type distance $d_{HT}$, introduced by Frosini et al. in \cite{frosinietal17},
as well as a slight variant of it, as our point of departure. A stability theorem is proven in
\cite{frosinietal17} for persistent homology (of sublevel set filtrations)
with respect to $d_{HT}$ for functional data defined on domains
with the same homotopy type. We combine $d_{HT}$ with a topological {\em cone construction}, 
used as a counterpart to barcode trimming at the level of functional spaces, that
has the virtue of making all domains contractible, so that the stability theorem of \cite{frosinietal17}
becomes applicable to functional data with arbitrary domains. Coning also
allows us to highlight the topology of regions where signals are strong and downplay topological
differences where signals are subdued. Previously,
Cerri et al.\,have used a related cone construction in the study of Betti numbers
for multidimensional persistent homology \cite{cerrietal08}.
We also define a metric on the space of $mm$-spaces that
downplays topological differences in regions far away from sizeable probability
mass and prove stability and consistency theorems for persistent homology
of centrality functions.

Functional and structural data on domains of different homotopy types are
commonplace in practice. For example, functional data acquired through imaging
often contain many topological defects in their domains that may require
a significant amount of image processing prior to analysis. The proposed
method allows us to easily deal with those defects and bypass such
pre-processing steps, provided that the defects do not occur near the ``support''
of the signal. It is also common for analysis of functional data to involve
delicate registration steps, especially in situations where images only partially
correspond. The present approach circumvents registration steps, still yielding
informative data summaries, provided that the ``supports'' of the signals
are fully captured by the images. In summary, if the images in a dataset capture the
regions where the important information resides, little image processing is
needed. 
%-------------------------
\begin{figure}[ht]
\begin{center}
\begin{tabular}{ccc}
\begin{tabular}{c}
\includegraphics[width=0.25\linewidth]{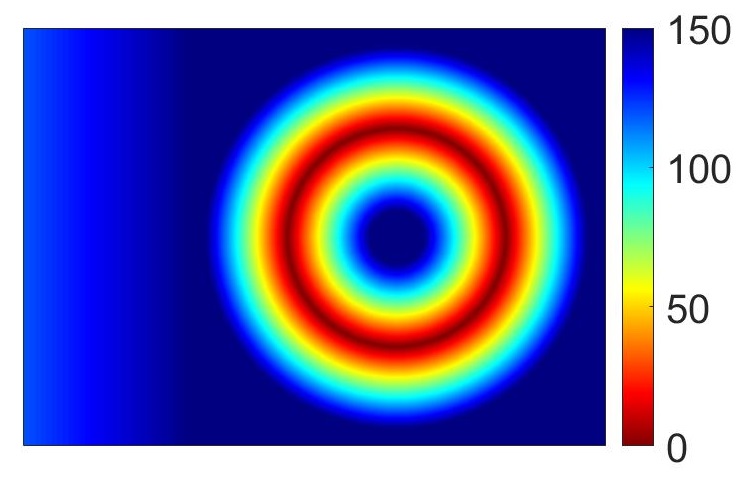} 
\end{tabular}
&
\begin{tabular}{c}
\includegraphics[width=0.25\linewidth]{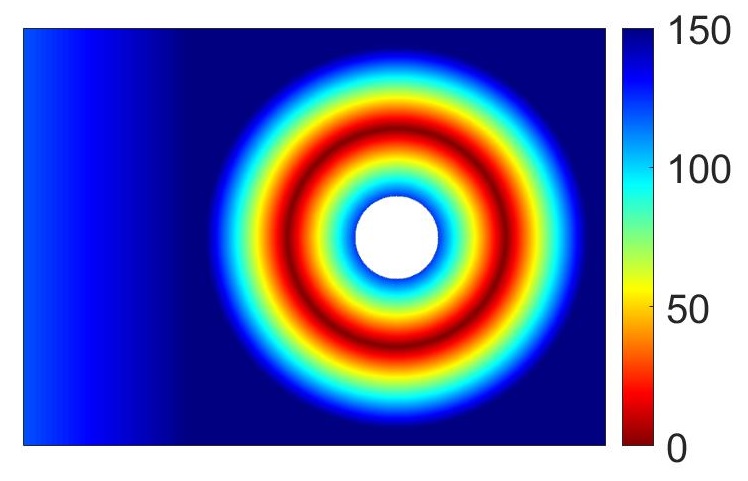}
\end{tabular}
&
\begin{tabular}{c}
\includegraphics[width=0.25\linewidth]{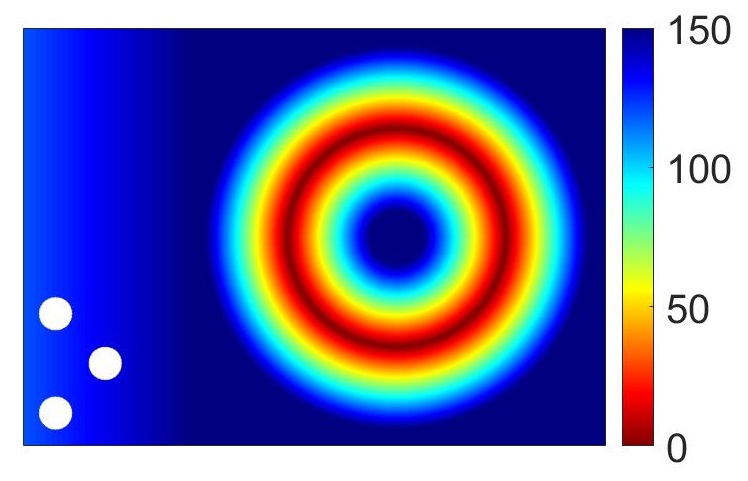} 
\end{tabular} \\
\begin{tabular}{c}
\includegraphics[width=0.25\linewidth]{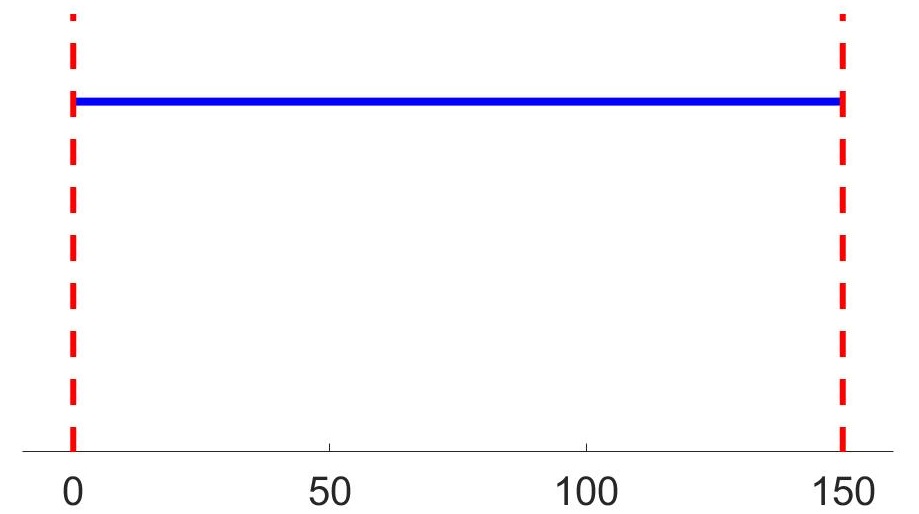} 
\end{tabular}
&
\begin{tabular}{c}
\includegraphics[width=0.25\linewidth]{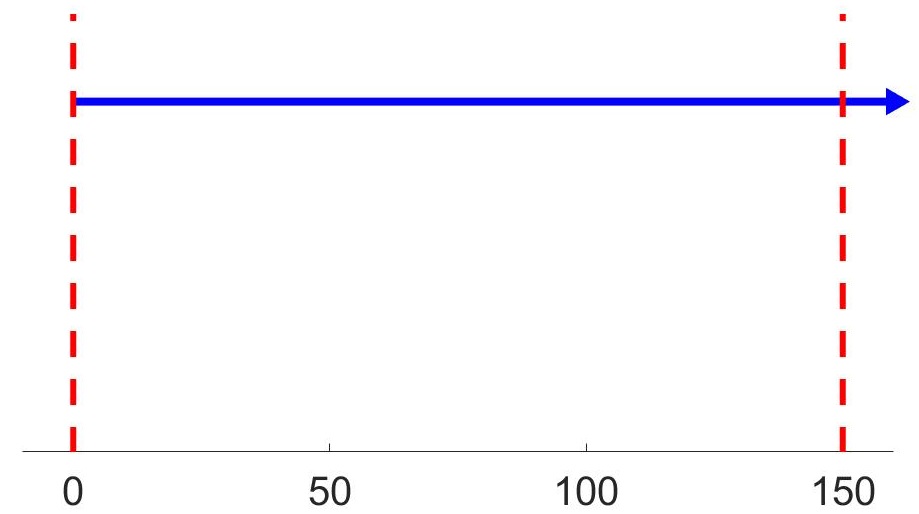}
\end{tabular}
&
\begin{tabular}{c}
\includegraphics[width=0.25\linewidth]{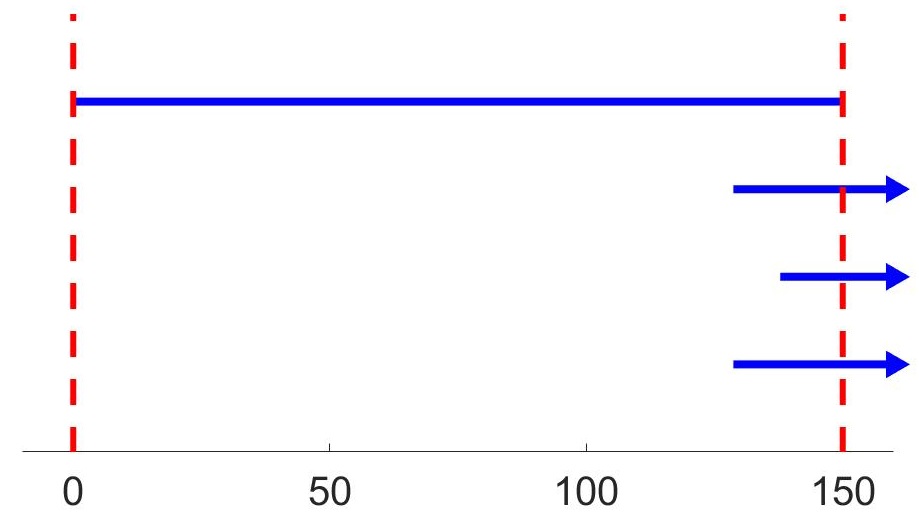} 
\end{tabular} \\
\begin{tabular}{c}
\includegraphics[width=0.25\linewidth]{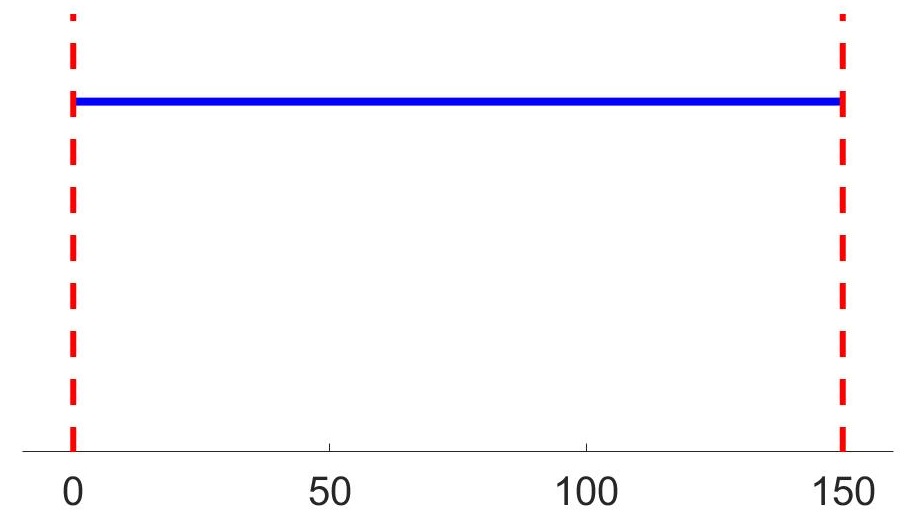} 
\end{tabular}
&
\begin{tabular}{c}
\includegraphics[width=0.25\linewidth]{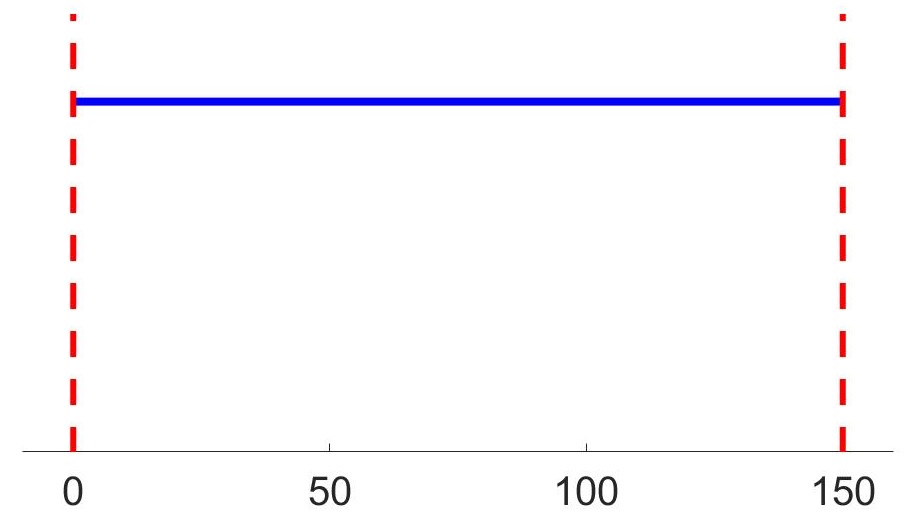}
\end{tabular}
&
\begin{tabular}{c}
\includegraphics[width=0.25\linewidth]{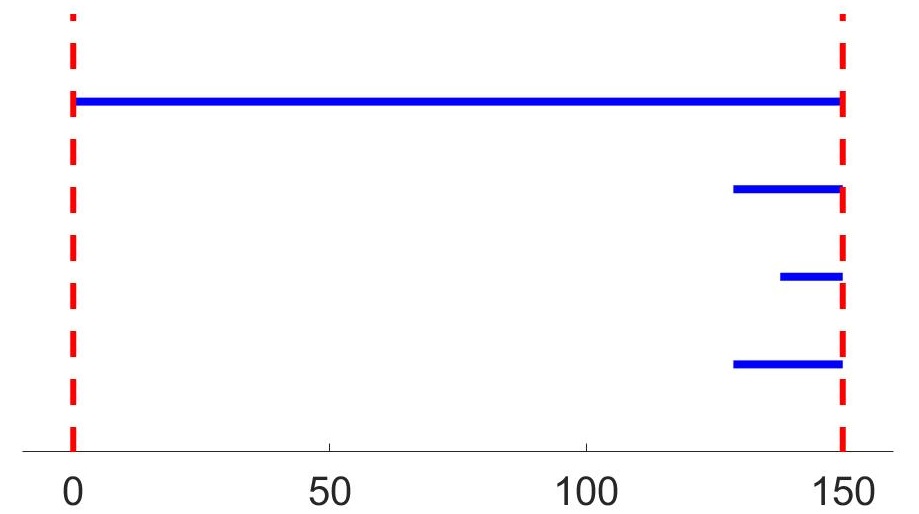} 
\end{tabular} \smallskip \\
(i) & (ii) & (iii)
\end{tabular}
\end{center}
\caption{The first row shows (i) a ``ground truth'' image on a rectangle $R$,
(ii) the image restricted to $R$ with an open disc removed, and (iii) the image with
three noisy holes in its domain. The second row panels display the corresponding
barcodes for 1-dimensional persistent homology, calculated using the images as filtering
functions. The third row shows the barcodes trimmed at the maximum values of the
functions. The arrows indicate bars of infinite length.}
\label{F:barcodes}
\end{figure}
%----
As an illustration, we simulate this type of data.
Fig.\,\ref{F:barcodes} shows signals with very similar shapes,
but defined on domains of different homotopy types. The first row shows heat
maps of (i) a ``ground truth'' function defined on a rectangle $R$ and with a
circular shape, (ii) the same function restricted to $R$ with an open ball removed
that is ``encircled'' by the signal, and (iii) the same function on $R$ with three noisy
holes located in a region where the signal is subdued. The function labeled (i)
has global minima along a circle and the center of that circle is a maximum.
(Recall that we made the convention that low values correspond to strong
signals.) The second row shows the barcodes for 1-dimensional
homology obtained from the sublevel set filtrations induced by the images.
The barcode for the ground truth image comprises a single bar of finite
persistence whose birth-death coordinates $(b,d)$ are the minimum and
maximum values of the function, respectively. These values are highlighted
by vertical dashed lines. For image (ii), we have a single bar with the same birth
coordinate $b$, but of infinite length because the signal wraps
around the hole in the domain. For image (iii), we have the original
bar of finite length in addition to three infinite bars whose birth coordinates
are close to $d$. Thus, in spite of having three signals of very similar shape, the
bottleneck distance between any pair of barcodes is infinite. The aforementioned
cone construction will have the effect of trimming the barcodes at $d$,
as shown on the third row. This yields three barcodes that lie close together
with respect to the bottleneck distance because the three noisy bars in (c)
are short lived.

It may be instructive to contrast the present approach to persistent homology of
metric measure spaces with the treatment by Blumberg et al. \cite{blumberg14}. A basic
philosophical difference is that whereas we define the persistent homology of a $mm$-space
$\mm{X}{\alpha}$ directly via centrality functions and prove stability and consistency
results, Blumberg et al. consider the pushforward of the product measure $\otimes_n \alpha$
on the $n$-fold product space $X \times \ldots \times X$ to barcode space via the map that
associates to a sample of size $n$, the persistent homology of its Vietoris-Rips complex,
proving stability and concentration results in this framework. Thus, their approach is based on
properties of the distribution of barcodes constructed from independent random draws
from the ``theoretical'' distribution $\alpha$.

As an application, we construct multi-scale persistent homology descriptors for
distributions $\alpha$ on a compact Riemannian manifold $M$. We denote the
geodesic distance on $M$ by $d_M$. Using diffusion distances $d_t$, $t >0$,
associated with the heat kernel on $M$ (cf.\,\cite{coifman,diazetal18}), we obtain metric
relaxations of $\mm{M}{\alpha} = (M, d_M, \alpha)$ via the 1-parameter family $(M, d_t, \alpha)$
of $mm$-spaces. (We recall the definition of $d_t$ in Section \ref{S:manifolds}.) We show
that this gives rise to a continuous path of persistence diagrams via the persistent homology
of their centrality functions. We employ the framework developed for the topological analysis
of $mm$-spaces to prove that this multi-scale topological descriptor is stable with respect
to the Wasserstein distance on the space of Borel measures on $(M, d_M)$.

The rest of the paper is structured as follows. Section \ref{S:ft} develops the
aforementioned metric on the space of functional topological spaces and
Section \ref{S:homology} proves the stability of persistence homology 
with respect to this metric. Section \ref{S:mm} addresses stability and
consistency of persistent homology of $mm$-spaces. Section \ref{S:manifolds}
is devoted to a multi-scale analysis of distributions on Riemannian manifolds
and Section \ref{S:final} closes the paper with some additional discussion.

\section{Functional Topological Spaces} \label{S:ft}

Throughout the paper, we assume that all topological spaces are compact. 
A {\em functional topological space} ($ft$-space) is a triple $\mm{X}{f} =
(X, \tau_X, f)$, where $f \colon X \to \real$ is a continuous function on the
topological space $(X,\tau_X)$.  Two $ft$-spaces $\mm{X}{f}$ and $\mm{Y}{g}$ are
isomorphic if there is a homeomorphism $h \colon X \to Y$ such that $f = g \circ h$. 
The collection of isomorphism classes of $ft$-spaces is denoted
$\fgh$. We abuse notation and denote an element of $\fgh$ as $\mm{X}{f}$.
We begin by reviewing a special case of the homotopy type distance $d_{HT}$ of
\cite{frosinietal17} that will be used in our study of functional data. We also introduce
a slight variant of $\prer$ that has better properties with respect to the proposed
cone construction, as explained in details below. We note that our notation and terminology
differ from that of \cite{frosinietal17}.

\subsection{The Homotopy Type Distance} \label{S:pmetric}

A {\em pairing } $(\phi, \psi)$ between  two topological spaces $(X, \tau_X)$ and
$(Y, \tau_Y)$ is a pair of continuous mappings $\phi \colon X \to Y$ and
$\psi \colon Y \to X$. We denote by $\Pi (X,Y)$ the collection
of all such pairings.

\begin{definition}
Let $h_0, h_1 \colon X \to X$ be continuous mappings
and $H \colon X \times [0,1] \to X$ a homotopy between
$h_0$ and $h_1$; that is, a continuous mapping such that
$H(x,0) = h_0 (x)$ and $H(x,1) = h_1 (x)$,
$\forall x \in X$. 
Given $\varepsilon > 0$ and a continuous function $f \colon X \to \real$, $H$ is
said to be an $\varepsilon$-homotopy from $h_0$ to $h_1$ over $f$ if 
\[
(f \circ H) (x,s)  \leq (f \circ h_0) (x) + \varepsilon \,,
\]
$\forall x \in X$ and $\forall s \in [0,1]$.
\end{definition}

Note that even if $H$ is an $\varepsilon$-homotopy from $h_0$ to $h_1$ over $f$, 
the mapping $\overline{H} (x,t) = H(x,1-t)$ may not be an $\varepsilon$-homotopy from
$h_1$ to $h_0$ over $f$.

\begin{definition} \label{D:pair}

Let $\mm{X}{f} = (X, \tau_X, f)$ and $\mm{Y}{g} = (Y, \tau_Y, g)$ be ft-spaces,
$\varepsilon > 0$, and $(\phi, \psi) \in \Pi(X,Y)$ a pairing between $\mm{X}{f} $ and $\mm{Y}{g}$.
\begin{itemize}
\item[(i)]  $(\phi, \psi)$ is an $\varepsilon$-pairing if $(g \circ \phi) (x) \leq f(x) + \varepsilon$ and
$(f \circ \psi) (y) \leq g(y) + \varepsilon$, $\forall x \in X$ and
$\forall y \in Y$.
\item[(i)]  $(\phi, \psi)$ is a strong $\varepsilon$-pairing if $|(g \circ \phi) (x) - f(x)| \leq \varepsilon$
and $|(f \circ \psi) (y) - g(y)| \leq \varepsilon$, $\forall x \in X$ and
$\forall y \in Y$.
\end{itemize}
Clearly, any strong $\varepsilon$-pairing is an $\varepsilon$-pairing.
\end{definition}

\begin{definition} \label{D:match}

Let $\mm{X}{f}$ and $\mm{Y}{g}$ be ft-spaces and $\varepsilon > 0$.

\begin{itemize}
\item[(i)] An $\varepsilon$-matching between $\mm{X}{f}$ and $\mm{Y}{g}$ is an
$\varepsilon$-pairing $(\phi, \psi)$ that satisfies:
\begin{itemize}
\item[(a)] there is a $2\varepsilon$-homotopy from $I_X$ to $\psi \circ \phi$ over $f$;
\item[(b)] there is a $2\varepsilon$-homotopy from $I_Y$ to $\phi \circ \psi$ over $g$,
\end{itemize}
where $I_X$ and $I_Y$ denote the identity maps of $X$ and $Y$, respectively.

\item[(ii)] A strong $\varepsilon$-matching between $\mm{X}{f}$ and $\mm{Y}{g}$ 
is a strong $\varepsilon$-pairing $(\phi, \psi)$ that satisfies (a) and (b) above.
\end{itemize}
\end{definition}
%---
\noindent
We use the notation $\mm{X}{f} \sim_\varepsilon \mm{Y}{g}$ to indicate that there
is an $\varepsilon$-matching between $\mm{X}{f}$ and $\mm{Y}{g}$.  The notation
 $\mm{X}{f} \approx_\varepsilon \mm{Y}{g}$ indicates the existence of a strong
 $\varepsilon$-matching. Note that $\exists \varepsilon > 0$ such that
 $\mm{X}{f} \sim_\varepsilon \mm{Y}{g}$ if and only if $X$ and $Y$ are homotopy
 equivalent. The same statement holds for strong matchings.

\begin{definition}

Let $\mm{X}{f}, \mm{Y}{g} \in \fgh$ be
$ft$-spaces.
\begin{itemize}
\item[(i)] \rm{(Frosini et al.\,\cite{frosinietal17})} The $\prer$ distance between
$\mm{X}{f}$ and $\mm{Y}{g}$ is defined as $\prer (\mm{X}{f}, \mm{Y}{g}) = \inf \{\varepsilon > 0 \,|\, 
\mm{X}{f} \sim_\varepsilon \mm{Y}{g}\}$, if $X$ and $Y$ are homotopy equivalent, and
$\prer (\mm{X}{f}, \mm{Y}{g}) = \infty$, otherwise.

\item[(ii)] The $\var$ distance is defined as $\var (\mm{X}{f}, \mm{Y}{g}) = \inf \{\varepsilon > 0 \,|\, 
\mm{X}{f} \approx_\varepsilon \mm{Y}{g}\}$, if $X$ and $Y$ are homotopy equivalent, and
$\var (\mm{X}{f}, \mm{Y}{g}) = \infty$, otherwise.
\end{itemize}
\end{definition}
%------
Both $\prer$ and $\var$ define extended pseudo-metrics on $\fgh$, extended meaning that
they may attain the value $\infty$. The triangle inequality is easily verified.

\subsection{A Cone Construction}

As explained in the Introduction, we introduce a cone construction
that gives a counterpart to barcode trimming at the $ft$-space level.
Let $(X, \tau_X)$ be a compact topological space. On the product
$X \times [0,1]$, consider the equivalence relation generated
by $(x, 1) \sim (x', 1)$, $\forall x, x' \in X$. The
cone on $X$ is the quotient space $C_X = X \times [0,1] / \sim$.
The quotient topology is denoted $\tau_{C_X}$.
The cone point is the equivalence class of any $(x,1)$, denoted
$\ast_X = [(x,1)]$. 

\begin{definition} \label{D:signal}
Let $\mm{X}{f} \in \fgh$ be a ft-space. The cone on $\mm{X}{f}$,
denoted $\cone(\mm{X}{f})$, is the ft-space $(C_X, \tau_{C_X}, f^\ast)$,
where $f^\ast \colon C_X \to \real$ is defined by
\[
f^\ast ([(x,t)]) = (1-t) f(x) + t m_f \,,
\]
for $x \in X$ and $0 \leq t \leq1$. Here, $m_f$ is the maximum
value of $f$.
\end{definition}

Note that for any $a < m_f$, the sublevel set $(C_X)_a = (f^\ast)^{-1} (-\infty, a]$
strong deformation retracts along cone lines to $X_a = f^{-1} (-\infty, a]$ and
therefore they have the same homology. Moreover, for $a \geq m_f$, $(C_X)_a = C_X$
which is contractible. Thus, any bar in a barcode whose birth coordinate is close
to $m_f$ will be short-lived, with the possible exception of a
single bar in $H_0$.

Define the {\em cone operator\ } $\cone \colon \fgh \to \fgh$ by
$\cone (\mm{X}{f}) := (C_X, \tau_{C_X}, f^\ast)$ and let $r_\infty$ and $\hat{r}_\infty$
be the pseudo-metrics on $\fgh$ induced by $\prer$ and $\var$, respectively, under
this operator. In other words, 
\begin{equation} \label{E:rinf}
r_\infty (\mm{X}{f}, \mm{Y}{g}) := \prer (\cone (\mm{X}{f}), \cone(\mm{Y}{g})) \,.
\end{equation}
and
\begin{equation} \label{E:rinf1}
\hat{r}_\infty (\mm{X}{f}, \mm{Y}{g}) := \var (\cone (\mm{X}{f}), \cone(\mm{Y}{g})) \,.
\end{equation}

\smallskip
\noindent
{\em Remark.} Since $C_X$ and $C_Y$ are contractible, thus homotopy
equivalent, we have that $r_\infty (\mm{X}{f}, \mm{Y}{g}) < \infty$ and
$\hat{r}_\infty (\mm{X}{f}, \mm{Y}{g}) < \infty$, for any
$\mm{X}{f}, \mm{Y}{g} \in \fgh$.
 
 \smallskip
 
One of the drawbacks in directly using the $\prer$ distance in our analysis of functional data
is that the cone operator, designed to simplify signals, may in fact increase the $\prer$ distance
between $ft$-spaces, making their dissimilarities even more pronounced, as illustrated
by the following example.
 
\smallskip
 
\noindent {\em Example.} Let $X = \{(0,0), (1,0)\}$ and
$Y = \{(0,y) \,|\, 0 \leq y \leq 1\} \cup \{(1,0)\}$ with the topology induced by the
Euclidean distance in $\real^2$. Let $\pi \colon \real^2 \to \real$ denote projection
onto the second coordinate and set $f = \pi |_X$ and $g= \pi |_Y$, so that $f \equiv 0$.
Then, one may show
 that $\prer (\mm{X}{f}, \mm{Y}{g}) = 0$ and $r_\infty (\mm{X}{f}, \mm{Y}{g}) = 
 \prer (\cone (\mm{X}{f}), \cone(\mm{Y}{g})) = 1$. 
 
 \smallskip
 
In spite of the obvious  inequality $\prer (\mm{X}{f}, \mm{Y}{g}) \leq \var (\mm{X}{f}, \mm{Y}{g})$,
coning exhibits a much better  behavior with respect to $\var$, making it a more natural 
metric to adopt in some situations. We close this section by showing
that the cone operator is non-expansive with respect to $\var$. 
 
 \begin{lemma} \label{L:max}
Let $\mm{X}{f}$ and $\mm{Y}{g}$ be ft-spaces. If there is a strong $\varepsilon$-pairing
between $\mm{X}{f}$ and $\mm{Y}{g}$, then
 $|m_f - m_g| \leq \varepsilon$.
 \end{lemma}
 \begin{proof}
Let $(\phi, \psi)$ be such that $|g(\phi(x)) - f(x)| \leq \varepsilon$ and
$|f(\psi(y)) - g(y)| \leq \varepsilon$, for all $x \in X$ and $y \in Y$. Pick
$x_0 \in X$ and $y_0 \in Y$ that satisfy $m_f = f(x_0)$ and $m_g = g (y_0)$.
If $m_g \leq m_f$, then, $g (\phi(x_0)) \leq m_g \leq m_f = f(x_0)$, so that
 $|m_f - m_g| \leq |f(x_0) - g (\phi(x_0))| \leq \varepsilon$. The same argument
 applies if $m_f \leq m_g$, proving the lemma.
 \end{proof}
 
Given a map $\phi \colon X \to Y$, we refer to
$\phi_c \colon C_X \to C_Y$, defined by $\phi_c ([(x,t)]) = [(\phi(x), t)]$,
as the cone on $\phi$. We use the notation $\phi_c$, instead of $\phi^\ast$,
to distinguish this construction from the cone on a signal $f$
(see  Definition \ref{D:signal}).
 
\begin{lemma} \label{L:cone}
Let $\mm{X}{f}, \mm{Y}{g} \in \fgh$ and $\varepsilon >0$. 
 
\begin{itemize}
\item[(i)] If $(\phi, \psi) \in \Pi(X,Y)$ is a strong $\varepsilon$-pairing between
$\mm{X}{f}$ and $\mm{Y}{g}$, then $(\phi_c, \psi_c) \in \Pi(C_X, C_Y)$ is
a strong $\varepsilon$-pairing between $\cone (\mm{X}{f})$ and $\cone (\mm{Y}{g})$.
\item[(ii)]  If $(\phi, \psi)$ is a strong $\varepsilon$-matching between
$\mm{X}{f}$ and $\mm{Y}{g}$, then $(\phi_c, \psi_c)$ is
a strong $\varepsilon$-matching between $\cone (\mm{X}{f})$ and $\cone (\mm{Y}{g})$.
\end{itemize} \end{lemma}
\begin{proof}
(i)  From the definitions of cones and Lemma \ref{L:max}, we have that
\begin{equation}
\begin{split}
|g^\ast (\phi_c ([(x,t)])) - f^\ast([(x,t)])| &\leq
(1-t) |(g \circ \phi) (x) - f(x)| + t |m_f - m_g|  \\
&\leq (1-t) \varepsilon + t \varepsilon  = \varepsilon \,.
\end{split}
\end{equation}
Similarly, $|f^\ast (\psi_c ([(x,t)])) - g^\ast([(x,t)])| \leq \varepsilon$,
proving the claim.

(ii) Let $(\phi, \psi)$ be a strong $\varepsilon$-matching between  $\mm{X}{f}$
and $\mm{Y}{g}$. By (i), $(\phi_c, \psi_c)$ is a strong $\varepsilon$-pairing, so it suffices
to verify the condition on homotopies. Let $H \colon X \times [0,1] \to X$ be a
$2 \varepsilon$-homotopy from $I_X$ to $\psi \circ \phi$ over $f$. We show that
$H_c \colon C_X \times [0,1] \to C_X$ defined by $H_c ([x,t],s) = [H(x,s),t]$ is
a $2\varepsilon$-homotopy from $I_{C_X}$ to $\psi_c \circ \phi_c$ over $f^\ast$.
Indeed,
\begin{equation}
\begin{split}
f^\ast(H_c ([x,t],s)) &= (1-t) f(H (x,s)) + t m_f \leq (1-t) (f(x) + 2 \varepsilon) + t m_f \\
&= f^\ast ([x,t]) + 2 (1-t) \varepsilon \leq f^\ast ([x,t]) + 2 \varepsilon \,.
\end{split}
\end{equation}
Similarly, we construct a $2\varepsilon$-homotopy from $I_{C_Y}$ to
$\phi_c \circ \psi_c$ over $g^\ast$. This concludes the proof.
 \end{proof}

\begin{proposition} \label{P:match}
The inequality
$\hat{r}_\infty (\mm{X}{f}, \mm{Y}{g}) \leq \var (\mm{X}{f}, \mm{Y}{g})$
holds for any $\mm{X}{f}, \mm{Y}{g} \in \fgh$.
\end{proposition}

\begin{proof}
The statement is trivial if $\var (\mm{X}{f}, \mm{Y}{g}) = \infty$, so we
assume that this distance is finite. Let $\varepsilon > \var (\mm{X}{f}, \mm{Y}{g})$.
Then, there is a strong $\varepsilon$-matching $(\phi,\psi)$ between $\mm{X}{f}$
and $\mm{X}{g}$. By Lemma \ref{L:cone} (ii), $(\phi_c, \psi_c)$ is a strong
$\varepsilon$-matching between $\cone (\mm{X}{f})$ and $\cone (\mm{Y}{g})$. Thus
$\hat{r}_\infty (\mm{X}{f}, \mm{Y}{g}) \leq \varepsilon$. Taking
infimum over $\varepsilon > \var (\mm{X}{f}, \mm{Y}{g})$, the claim follows.
\end{proof}

%--------------------------------------

\section{Topology of Functional Spaces} \label{S:homology}

We briefly recall the definitions of persistence modules over
$(\real, \leq)$ and interleaving distance between two persistence modules.
For more details, we refer the reader to \cite{chazaletal,lesnick11}. We regard $(\real, \leq)$ as
a category $\mathscr{R}$ whose objects are the points $a \in \real$ with a single morphism
$a \to b$ if $a \leq b$ and none otherwise. A persistence module $\mathbb{V}$
(over a fixed field $\mathbb{F}$) is a functor $\mathcal{F}$ from $\mathscr{R}$ to the
category of vector spaces over $\mathbb{F}$. We use the notation $V_a := \mathcal{F} (a)$ for the vector
space over $a \in \real$. For $a \leq b$, we write $\nu_a^b \colon V_a \to V_b$ for the
linear mapping associated with the morphism $a \to b$.

Let $\mathbb{V} = \{V_a, \nu_a^b\}$ and $\mathbb{W} = \{W_a, \omega_a^b\}$ 
be persistence modules and $\varepsilon > 0$. A morphism $\Phi \colon
\mathbb{V} \to \mathbb{W}$ of degree $\varepsilon$ is a collection of
linear mappings $\phi_a \colon V_a \to W_{a + \varepsilon}$, $a \in \real$,
satisfying $\omega_{a + \varepsilon}^{a + r +\varepsilon} \circ \phi_a
= \phi_{a+r} \circ \nu_a^{a + r}$, for any $a \in \real$ and $r \geq 0$. An
$\varepsilon$-interleaving between $\mathbb{V}$ and $\mathbb{W}$ is a pair
$\Phi \colon \mathbb{V} \to \mathbb{W}$ and $\Psi \colon \mathbb{W} \to \mathbb{V}$
of morphisms of degree $\varepsilon$ satisfying $\psi_{a+\varepsilon} \circ 
\phi_a = \nu_a^{a+2 \varepsilon}$ and $\phi_{a+\varepsilon} \circ 
\psi_a = \omega_a^{a+2 \varepsilon}$, $\forall a \in \real$. We write
$\mathbb{V} \sim_\varepsilon \mathbb{W}$ to indicate that there is an
$\varepsilon$-interleaving between the two persistence modules. The interleaving
distance $d_I$ is defined as:
\begin{itemize}
\item[(i)] $d_I (\mathbb{V}, \mathbb{W}) = \infty$, if no interleaving exists;
\item[(ii)] $d_I (\mathbb{V}, \mathbb{W}) = \inf \{\varepsilon > 0 \,| \,
\mathbb{V} \sim_\varepsilon \mathbb{W} \}$, otherwise.
\end{itemize}

A persistence module $\mathbb{V}$ is {\em tame} if $\nu_a^b$ has finite rank,
for any $a < b$. There is a well-defined persistence diagram (or barcode),
denoted $PD(\mathbb{V})$, associated with any tame $\mathbb{V}$ and the
Isomorphism Theorem for persistence modules states that
\begin{equation} \label{E:iso}
d_b (PD(\mathbb{V}), PD(\mathbb{W})) = d_I (\mathbb{V}, \mathbb{W}) \,,
\end{equation}
if $\mathbb{V}$ and $\mathbb{W}$ are tame, where $d_b$ denotes
bottleneck distance \cite{cohenetal07,desilvaetal}.

\subsection{Stability of Persistent Homology} \label{S:stability}

Let $\mm{X}{f} \in \fgh$ be an $ft$-space. For $a \in \real$, let
$X_a = f^{-1} (-\infty, a]$ be the corresponding sub-level set of $f$.
Clearly, $X_a \subseteq X_b$, if $a \leq b$, and $\cup_{a \in \real} X_a = X$.
Thus, the sub-level sets of $f$ induce a filtration of $X$ by closed subsets, which may be
viewed as a functor from $\mathcal{R}$ to the category of topological spaces and
continuous mappings. Composition with the $i$-dimensional
homology functor (with coefficients in $\mathbb{F}$) yields a
persistence module, which we denote by $\homol_i (\mm{X}{f})$. The
vector space over $a \in \real$ is $H_i (X_a)$ and the morphism
$\iota_a^b \colon H_i(X_a) \to H_i (X_b)$, $a \leq b$, is the homomorphism
on homology induced by inclusion.

A stability theorem for persistent homology of $ft$-spaces has been
proven in \cite{frosinietal17}, but we include a proof for the one-parameter persistence
case needed in this paper.

\begin{theorem}[Frosini et al.\,\cite{frosinietal17}] \label{T:stab}
Let $\mm{X}{f}$ and $\mm{Y}{g}$ be functional topological spaces. Then,
\[
d_I (\homol_i (\mm{X}{f}), \homol_i (\mm{Y}{g})) \leq
\prer (\mm{X}{f}, \mm{Y}{g}) \,.
\]
\end{theorem}
\begin{proof}
It suffices to consider the case $\prer (\mm{X}{f}, \mm{Y}{g}) < \infty$.
Let $(\phi, \psi) \in \Pi (X,Y)$ be an $\varepsilon$-matching between 
$\mm{X}{f}$ and $\mm{Y}{g}$. For any $a \in \real$, condition (i) in
Definition \ref{D:pair} ensures that $\phi (X_a) \subseteq Y_{a + \varepsilon}$.
Thus, $\phi$ induces a mapping $\phi_a \colon X_a \to Y_{a + \varepsilon}$.
Similarly, $\psi$ induces mappings $\psi_a \colon Y_a \to X_{a + \varepsilon}$.
Condition (a) in Definition \ref{D:match}(i) implies that $\psi_{a + \varepsilon}
\circ \phi_a \colon X_a \to X_{a + 2 \varepsilon}$ is homotopic to the inclusion
map $X_a \hookrightarrow X_{a + 2 \varepsilon}$. Analogously,
$\phi_{a + \varepsilon} \circ \psi_a \colon Y_a \to Y_{a + 2 \varepsilon}$ is
homotopic to the inclusion map $Y_a \hookrightarrow Y_{a + 2 \varepsilon}$.
Thus, the 
homomorphisms on homology induced by $\phi_a$ and $\psi_a$, $a \in \real$,
yield an $\varepsilon$-interleaving between $\homol_i (\mm{X}{f})$ and 
$\homol_i (\mm{Y}{g})$. This implies that 
$d_I (\homol_i (\mm{X}{f}), \homol_i (\mm{Y}{g})) \leq \varepsilon$. Taking
infimum over $\varepsilon > \prer (\mm{X}{f}, \mm{Y}{g})$, the result follows.
\end{proof}

\begin{corollary} \label{C:stab}
Let $\mm{X}{f}$ and $\mm{Y}{g}$ be functional topological spaces. Then,
\[
d_I (\homol^\ast_i (\mm{X}{f}), \homol^\ast_i (\mm{Y}{g})) \leq
r_\infty (\mm{X}{f}, \mm{Y}{g}) \leq  \hat{r}_\infty (\mm{X}{f}, \mm{Y}{g})
\leq \var (\mm{X}{f}, \mm{Y}{g}) \,,
\]
where $\homol^\ast_i (\mm{X}{f}) = \homol_i (\cone (\mm{X}{f}))$ and
$\homol^\ast_i (\mm{Y}{g}) = \homol_i (\cone (\mm{Y}{g}))$.
\end{corollary}
\begin{proof}
The first inequality follows from Theorem \ref{T:stab} and \eqref{E:rinf}. The second inequality 
is a direct consequence of the definition of the metrics $r_\infty$ and $\hat{r}_\infty$.
The third inequality was proven in Proposition \ref{P:match}.
\end{proof}

If $(X, \tau_X)$ is a triangulable space, then $\homol^\ast_i (\mm{X}{f})$
is tame (cf.\,\cite{cohenetal07,desilvaetal}) so that there is a well-defined persistence diagram
$PD_i^\ast (\mm{X}{\alpha})$ associated with $\homol^\ast_i (\mm{X}{f})$.

\begin{corollary} \label{C:diagram}
Let $\mm{X}{f}$ and $\mm{Y}{g}$ be triangulable ft-spaces and $i \geq 0$
be an integer. Then,
\[
d_B (PD_i^\ast (\mm{X}{f}), PD_i^\ast (\mm{Y}{g}) \leq
r_\infty (\mm{X}{f}, \mm{Y}{g}) \leq  \hat{r}_\infty (\mm{X}{f}, \mm{Y}{g})
\leq \var (\mm{X}{f}, \mm{Y}{g}) \,.
\]
\end{corollary}

\begin{proof}
This follows from Corollary \ref{C:stab} and \eqref{E:iso}.
\end{proof}

%--------------------------------------
\subsection{Effect of Coning on Persistent Homology}

As explained in the Introduction, one of the practical motivations for coning functional data 
is the truncation effect it has on persistent homology of
their sublevel set filtrations. Here, we show more formally that this is indeed the
effect of coning an $ft$-space.

\begin{definition}
Let $\mathbb{V} = \{V_a, \nu_a^b\}$ be a persistence module and $r \in \real$.
The $r$-truncation of $\mathbb{V}$ is defined as the persistence module
$T_r (\mathbb{V}) := \{\overline{V}_a, \overline{\nu}_a^b\}$, where
\[
\overline{V}_a =
\begin{cases}
V_a & \text{if $a<r$;} \\
0 & \text{if $a \geq r$;} \\
\end{cases}
\qquad \text{and} \qquad
\overline{\nu}_a^b=
\begin{cases}
\nu_a^b & \text{if $b<r$;} \\
0 & \text{if $b \geq r$.} \\
\end{cases}
\]
\end{definition}

Let $\mm{X}{f}$ be an $ft$-space and $m = \min f$. We denote by $\mm{P}{g}$ the
$ft$-space where $P = \{p\}$ is a  one-point space and $g \colon P \to \real$ is
given by $g(p) = m$. By the definition of $g$, the (constant) maps $\phi_1 \colon X \to P$
and $\phi_2 \colon C_X \to P$ have the property that the sublevel sets  of $f$, $f^\ast$ and $g$
satisfy $\phi_1 (X_a) \subseteq P_a$ and $\phi_2 ((C_X)_a) \subseteq P_a$, $\forall a \in \real$. 
It is simple to verify that, for each $i \geq 0$, $\phi_1$ and $\phi_2$ induce homomorphisms
$\phi_{1\ast} \colon \homol_i (\mm{X}{f}) \to \homol_i (\mm{P}{g})$ and 
$\phi_{2\ast} \colon \homol^\ast_i (\mm{X}{f}) \to \homol_i (\mm{P}{g})$ of persistence
modules. Note that $\homol_0  (\mm{P}{g})$ is isomorphic to the interval
module $\mathbb{I}_m^\infty$ associated with the interval $[m, \infty)$,
and $\homol_i  (\mm{P}{g}) = 0$ for $i >0$. The
kernels of $\phi_{1\ast}$ and $\phi_{2\ast}$ give persistence submodules
$\widetilde{\homol}_i (\mm{X}{f}) := \ker (\phi_{1\ast})$ and 
$\widetilde{\homol}^\ast_i (\mm{X}{f}) := \ker (\phi_{2\ast})$ of $\homol_i (\mm{X}{f})$
and $\homol^\ast_i (\mm{X}{f})$, respectively, which yield direct sum decompositions
\begin{equation} \label{E:decomposition}
\homol_i (\mm{X}{f}) \cong \widetilde{\homol}_i (\mm{X}{f}) \oplus  \homol_i (\mm{P}{g})
\quad \text{and} \quad
\homol^\ast_i (\mm{X}{f}) \cong \widetilde{\homol}^\ast_i (\mm{X}{f}) \oplus
 \homol_i (\mm{P}{g}) \,.
\end{equation}
(These decompositions are non-trivial only for the case when $i=0$ since $\homol_i (\mm{P}{g}) = 0$
for $i>0$.)
One may construct such isomorphisms by splitting the homomorphisms $\phi_{1\ast}$
and $\phi_{2\ast}$, as follows. Pick $x_0 \in X$ such that $f(x_0) = m$ and
let $\psi_1 \colon P \to X$ and $\psi_2 \colon P \to C_X$ be given by $\psi_1 (p) = x_0$
and $\psi_2 (p) = [(x_0,0)]$. Then, the induced homomorphisms $\psi_{1\ast}$ and
$\psi_{2\ast}$ on persistence modules are well defined and split $\phi_{1\ast}$ and
$\phi_{2\ast}$, as desired. As usual, the splitting is not natural.

\begin{proposition} \label{P:truncate}
Let $i \geq 0$ be an integer, $\mm{X}{f} \in \fgh$, and $m_f = \max f$. Then,
$\widetilde{\homol}^\ast_i (\mm{X}{f}) = \widetilde{\homol}_i (\cone (\mm{X}{f}))$ is
isomorphic to $T_{m_f} (\widetilde{\homol}_i (\mm{X}{f}))$.
\end{proposition}

\begin{proof}

Let $\iota \colon X \hookrightarrow C_X$ be the inclusion $x \mapsto [(x,0)]$. Then,
the sublevel sets of $f$ and $f^\ast$ satisfy $\iota (X_a) \subseteq (C_X)_a$, $\forall a \in \real$,
and $\iota$ induces a homomorphism $\iota_\ast \colon \homol_i (\mm{X}{f}) \to 
\homol_i (\cone(\mm{X}{f}))$.
The commutativity of the diagram
\begin{equation}
\xymatrix{
X_a \ar[drrr]^{\phi_1} \ar[dd]_{\iota} & & & \\
& & & P = \{p\} \\
(C_X)_a \ar[urrr]_{\phi_2} & & &
}
\end{equation}
for any $a \in \real$, implies that $\iota_\ast (\widetilde{\homol}_i (\mm{X}{f}))
\subseteq \widetilde{\homol}_i (\cone (\mm{X}{f}))$. Since $\iota$ induces a homotopy
equivalence between $X_a$ and $(C_X)_a$, for $a < m_f$, $\iota_\ast$ induces an isomorphism
$\iota_\ast \colon \widetilde{H}_i (X_a) \to \widetilde{H}_i ((C_X)_a)$, for $a < m_f$. Combining this 
with the fact that the reduced homology of $(C_X)_a$ is trivial for $a \geq m_f$, we
conclude that $\iota_\ast$ induces an isomorphism between 
$\widetilde{\homol}_i (\cone (\mm{X}{f}))$ and $T_{m_f} (\widetilde{\homol}_i (\mm{X}{f}))$,
as claimed.
\end{proof}

\begin{corollary}
Suppose that $\mm{X}{f} \in \fgh$ with $(X, \tau_X)$ triangulable and let 
$m = \min f$ and $m_f = \max f$.
For any integer $i \geq 0$, let $I_\lambda \subset \real$, $\lambda \in \Lambda_i$, be intervals
such that the corresponding interval modules $\mathbb{I}_\lambda$ yield a decomposition
$\widetilde{\homol}_i (\mm{X}{f}) = \bigoplus_{\lambda \in \Lambda_i}
\mathbb{I}_\lambda$. Then,
\begin{itemize}
\item[(i)] $\homol_0^\ast (\mm{X}{f}) \cong \mathbb{I}_m^\infty \oplus \bigoplus_{\lambda \in \Lambda_0}
\mathbb{I}_\lambda^{m_f}$;
\item[(ii)] $\homol_i^\ast (\mm{X}{f}) \cong \bigoplus_{\lambda \in \Lambda_i}
\mathbb{I}_\lambda^{m_f}$, for $i > 0$.
\end{itemize}
Here, $\mathbb{I}_\lambda^{m_f}$ and $\mathbb{I}_m^\infty$ are the interval modules associated
with $I^{m_f}_\lambda = I_\lambda \cap (-\infty, m_f)$ and $I_m^\infty = [m, \infty)$,
respectively. We make the convention that $\mathbb{I}_\lambda^{m_f}$ is trivial
if $I^{m_f}_\lambda = \emptyset$.
\end{corollary}

\begin{proof}
This is an immediate consequence of \eqref{E:decomposition} and Proposition \ref{P:truncate}.
\end{proof}
\vspace{0.2in}

% \hrule

% \vspace{0.2in}

% We now prove that the cone operator acts on persistent barcodes in the manner we described in the introduction. 

% First we analyze the effect of the cone operator on the merge tree \cite{morozov} associated to $\mathbb{X}_f =(X,\tau_X,f)$. The first remark is that since we do not assume that $X$ is connected, then the merge tree of $f$ may actually be a forest. Next, it is clear that the merge tree of $C\mathbb{X}_f$ will consist of a join all the trees in the forest corresponding to $\mathbb{X}_f$ exactly at level $m_f$. This will directly imply relationship between $PD_0(\mathbb{X}_f)$ and   $PD_0^\ast(\mathbb{X}_f)$. The following proposition provides a general relationship between the persistent diagrams of $\mathbb{X}_f$ and those of $C\mathbb{X}_f$.
 
%  \begin{proposition}
% Let $\mathbb{X}_f$ be any ft-space.  Then $\mathbb{H}_i^\ast(\mathbb{X}_f)$ can be obtained from 
% $\mathbb{H}_i(\mathbb{X}_f)$ as follows: 
% \begin{itemize}
% \item[$i=0$:] $a$
% \item[$i\geq 1$:] $a$
% \end{itemize}
% \facundo{continuar}
% \end{proposition}

%----------------

\section{Metric Measure Spaces} \label{S:mm}

A metric measure space ($mm$-space)  is a triple $\mm{X}{\alpha} =
(X, d_X, \alpha)$, where $\mathbb{X} = (X,d_X)$ is a metric space and
$\alpha$ is a Borel probability measure on $X$. We assume that
all $mm$-spaces are compact. Two $mm$-spaces
$\mm{X}{\alpha} = (X, d_X, \alpha)$ and
$\mm{Y}{\beta} = (Y, d_Y, \beta)$ are isomorphic if there is
an isometry $h \colon X \to Y$ such that $h_\ast (\alpha) =
\beta$. We denote by $\gw$ the collection of isomorphism
classes of compact $mm$-spaces, but often refer to $\mm{X}{\alpha}$
as an element of $\gw$. We begin
by equipping $\gw$ with a pseudo-metric that let us address
stability of persistence diagrams that seek to emphasize the shape
of regions rich in probability mass, downplaying the
geometry of regions that are not central to the distribution, as
characterized by larger values of the centrality function.

\subsection{A Metric on $\boldsymbol{\gw}$}

To motivate our definition of a (pseudo) metric on $\gw$ that is well suited to
the study of stability of the topology of centrality functions, we briefly digress
on well known metrics of related nature. Metrics such as the Gromov-Hausdorff distance
$d_{GH}$ between compact metric spaces and the Gromov-Wasserstein distance
$d_{GW,p}$ beween metric measure spaces may be defined via measures of distortion
of correspondences between the relevant spaces. ``Hard'' correspondences
given by a relation between two metric domains may be used to define
$d_{GH}$, whereas ``soft'' correspondences given by couplings of 
probability measures may be used for $d_{GW,p}$. More precisely, let
$\mathbb{X} = (X, d_X)$ and $\mathbb{Y} = (Y, d_Y)$ be compact
metric spaces. A {\em correspondence} $R$ between $X$ and $Y$ is
a subset $R \subset X \times Y$ such that $\pi_X (R) = X$ and
$\pi_Y (R) = Y$. The distortion of $R$ is given by
\begin{equation}
D (R) := \sup_{\tiny \begin{matrix}(x,y) \in R \\ (x',y') \in R \end{matrix}}
|d_X(x,x') - d_Y (y,y')| \,,
\end{equation}
and the Gromov-Hausdorff distance may be defined as
\begin{equation}
d_{GH} (\mathbb{X}, \mathbb{Y}) := \frac{1}{2}
\inf_{\tiny R} D (R)\,,
\end{equation}
where the infimum is taken over all correspondences $R$ between
$X$ and $Y$ \cite{gromov}. 
Similarly, a {\em coupling} between $\mm{X}{\alpha}$ and
$\mm{Y}{\beta}$ is a Borel probability measure $\mu$ on $X \times Y$
that marginalizes to $\alpha$ and $\beta$, respectively. In other words,
$(\pi_X)_\ast (\mu) = \alpha$ and $(\pi_Y)_\ast (\mu) = \beta$.
The collection of all such couplings is denoted
$\Gamma (\mm{X}{\alpha}, \mm{Y}{\beta})$.
For $p \geq 1$, the $p$-distortion of a coupling $\mu$ is defined as
\begin{equation}
D_p (\mu) := \left( \int \int
|d_X (x,x') - d_Y (y,y')|^p \, d\mu(x,y) d\mu (x', y') \right)^{1/p}
\end{equation}
and
\begin{equation}
d_{GW,p} (\mm{X}{\alpha}, \mm{Y}{\beta}) := \inf_\mu D_p (\mu),
\end{equation}
with the infimum taken over all
$\mu \in \Gamma (\mm{X}{\alpha}, \mm{Y}{\beta})$ (cf.\,\cite{memoli2011}).
In both definitions, a single correspondence or coupling is used for
both $(x,y)$ and $(x',y')$. We employ a hybrid version that combines a hard
correspondence for $(x,y)$ and a soft correspondence for $(x', y')$, except that
the hard correspondence  will be given by a pair of continuous mappings
between metric spaces rather than a relation. As a first step toward this goal,
for each  coupling $\mu \in \Gamma (\mm{X}{\alpha}, \mm{Y}{\beta})$,
we define a (pseudo) metric $d_{p, \mu}$ on the disjoint union $X \sqcup Y$,
as follows.

On the $X$ component, we define a pseudo-metric $d_{p, \alpha}$ by mapping
$X$ to the function space $\lp (\mm{X}{\alpha})$ and taking the induced metric,
as follows.
Let $K \colon (X, d_X) \to \lp (\mm{X}{\alpha})$, $x \mapsto K_x$,
be given by  $K_x (x') = d_X (x, x')$. Then, $d_{p, \alpha}$ is the (pseudo) metric
on $X$ induced by the norm $\| \cdot \|_{p, \alpha}$ under this
map. In other words,
\begin{equation}
\begin{split}
d_{p, \alpha} (x_1, x_2) = 
\left( \int_X |d_X (x_1,x') - d_X (x_2, x')|^p \, d \alpha (x') \right)^{1/p} \,.
\end{split}
\end{equation}
Note that, since $\mu$ has $\alpha$ as marginal, we have
\begin{equation} \label{E:marginal}
\begin{split}
d_{p, \alpha} (x_1, x_2) &= 
\left( \int_{X\times Y} |d_X (x_1,x') - d_X (x_2, x')|^p \, d \mu (x',y') \right)^{1/p} \\
&= \|K_{x_1} - K_{x_2}\|_{p, \mu} \,.
\end{split}
\end{equation}
Similarly, define $d_{p, \beta}$ on the $Y$ component. Using the coupling,
we define $d_{p, \mu}$ on the disjoint union $X \sqcup Y$ that restricts to
$d_{p,\alpha}$ and $d_{p,\beta}$ on $X$ and $Y$, respectively. The distance
between $x \in X$ and $y \in Y$ is given by
\begin{equation} \label{E:distmu}
\begin{split}
d_{p,\mu} (x,y) &=  \left( \int_{X \times Y}
|d_X(x, x') - d_Y (y, y')|^p \, d\mu (x',y') \right)^{1/p} \\
&= \|K_{x} - K_{y}\|_{p, \mu}
\end{split}
\end{equation}

\begin{proposition}
For any $p \geq 1$ and $\mu \in \Gamma (\mm{X}{\alpha},\mm{Y}{\beta})$,
the distance $d_{p,\mu}$ defines a pseudo-metric on $X \sqcup Y$.
\end{proposition}
\begin{proof}
The distance $d_{p,\mu}$ is clearly symmetric, so it suffices to verify the
triangle inequality, which is valid because $d_{p,\mu}$  is induced by an
$\lp$-norm. Indeed, suppose $x_1, x_2 \in X$ and $y \in Y$. Then, by
\eqref{E:marginal}, \eqref{E:distmu} and the Minkowski inequality, we have that
\begin{equation}
\begin{split}
d_{p,\mu} (x_1,y) = \|K_{x_1} - K_{y}\|_{p, \mu} &\leq \|K_{x_1} - K_{x_2}\|_{p, \mu}
+ \|K_{x_2} - K_{y}\|_{p, \mu} \\
&= d_{p,\mu} (x_1,x_2) + d_{p,\mu} (x_2,y) \,.
\end{split}
\end{equation}
Other cases may be verified in an analogous manner.
\end{proof}

\begin{definition} \label{D:distortion}
Let $\mm{X}{\alpha}, \mm{Y}{\beta}$ be $mm$-spaces, $(\phi, \psi) \in \Pi (X,Y)$
a pairing between $(X,d_X)$ and $(Y, d_Y)$, and $\mu \in
\Gamma (\mm{X}{\alpha},\mm{Y}{\beta})$ a coupling. We define three distortions
associated with $\phi$, $\psi$ and $\mu$:
\begin{itemize}
\item[(i)] $D(\phi, \mu) := \sup_{x \in X} d_{p, \mu} (x, \phi (x))$ and
$D(\psi, \mu) := \sup_{y \in Y} d_{p, \mu} (\psi(y), y)$;
\item[(ii)] $D(\phi, \psi, \mu) := \max \{D(\phi, \mu), D(\psi, \mu)\}$.
\end{itemize}
\end{definition}

\begin{definition} \label{D:match1}
Let $\mm{X}{\alpha}, \mm{Y}{\beta}$ be $mm$-spaces, $\varepsilon > 0$ and
$p \geq 1$. An $\varepsilon$-matching (of order $p$) between
$\mm{X}{\alpha}$ and $\mm{Y}{\beta}$ is a pairing $(\phi, \psi) \in \Pi (X,Y)$
and a coupling $\mu \in \Gamma (\alpha,\beta)$ such that:
\begin{itemize}
\item[(i)] $D(\phi, \psi, \mu) \leq \varepsilon$;
\item[(ii)] There is a $2 \varepsilon$-homotopy from the identity map $I_{C_X}$ to
$\psi_c \circ \phi_c \colon C_X \to C_X$ over
$\sigma^\ast_{p,\mm{X}{\alpha}} \colon C_X \to \real$;
\item[(iii)]  There is a $2 \varepsilon$-homotopy from the identity map $I_{C_Y}$ to
$\phi_c \circ \psi_c \colon C_Y \to C_Y$ over $\sigma^\ast_{p, \mm{Y}{\beta}} \colon C_Y \to \real$.
\end{itemize}
We write $\mm{X}{\alpha} \sim_\varepsilon \mm{Y}{\beta}$ to indicate
that there is an $\varepsilon$-matching between $\mm{X}{\alpha}$ and $\mm{Y}{\beta}$
and define
\begin{equation}
\delta_p (\mm{X}{\alpha}, \mm{Y}{\beta}) = 
\inf \{ \varepsilon \,|\,  \mm{X}{\alpha} \sim_\varepsilon \mm{Y}{\beta}\} \,.
\end{equation}
For any $p \geq 1$, $\delta_p$ induces a pseudo-metric on $\gw$.

\end{definition}

The next lemma shows that, for probability measures defined on the same
metric space, the Wasserstein distance gives an upper bound for $\delta_p$.
We denote the collection of all Borel probability measures on a compact metric space
$(X, d_X)$ by $\mathcal{P} (X)$ and the Wasserstein metric of order $p \geq 1$
on $\mathcal{P} (X)$ by $w_p$. 
%---------
\begin{lemma} \label{L:wass}
For any $\alpha, \beta \in \mathcal{P} (X)$,
$\delta_p (\mm{X}{\alpha}, \mm{X}{\beta}) \leq w_p (\alpha, \beta)$.
\end{lemma}
\begin{proof}
Let $\mu \in \Gamma (\alpha, \beta)$ be a coupling that realizes
$w_p (\alpha, \beta)$; that is, 
\begin{equation}
\left( \int_X d^p_X (x,x') \, d \mu (x, x') \right)^{1/p}= w_p (\alpha, \beta) \,.
\end{equation}
By the triangle inequality, the trivial pairing $(\phi, \psi) = (I_X, I_X)$ between
$\mm{X}{\alpha}$ and $\mm{X}{\beta}$ satisfies
\begin{equation}
\begin{split}
d_{p, \mu} (x, \phi(x)) &=
\left( \int_{X \times X}
|d_X(x, x_1) - d_X (x, x_2)|^p \, d\mu (x_1, x_2) \right)^{1/p} \\
&\leq \left( \int_{X \times X}
d^p_X(x_1, x_2)  \, d\mu (x_1, x_2) \right)^{1/p} =
w_p (\alpha, \beta) \,.
\end{split}
\end{equation}
Thus, condition (i) in Definition \ref{D:match1} is satisfied for
$\varepsilon = w_p (\alpha, \beta)$. Conditions (ii) and (iii)
are trivially satisfied by this pairing. Therefore, 
%\begin{equation}
$\delta_p (\mm{X}{\alpha}, \mm{X}{\beta})
\leq w_p (\alpha, \beta)$.
%\end{equation}
\end{proof}

\subsection{Stability and Consistency} \label{S:consistency}

One of our main goals is to study the shape of a
$mm$-space $\mm{X}{\alpha}$ via the persistent homology of the sublevel set
filtration of $C_X$ induced by the cone on the centrality function. As an
intermediate step, we map $(\gw, \delta_p)$ to the functional space
$(\fgh, r_\infty)$. For $p \geq 1$, let $\ecc{p} \colon \gw \to \fgh$
be the mapping given by
\begin{equation}
\ecc{p} (\mm{X}{\alpha}) = (X, \tau_X, \sigma_{p,\mm{X}{\alpha}}) \,,
\end{equation}
where $\tau_X$ is the topology underlying the metric space $(X, d_X)$ and
$\sigma_{p,\mm{X}{\alpha}}$ is the $p$-centrality function of $\mm{X}{\alpha}$.
We denote by $\homol^\ast_i (\Theta_p (\mm{X}{\alpha}))$
the persistence module for $i$-dimensional homology of the
cone $\cone (\ecc{p} (\mm{X}{\alpha}))$. If $(X, d_X)$ is triangulable,
there is a well defined persistence diagram
$PD^\ast_i (\Theta_p (\mm{X}{\alpha}))$ associated with this persistence module.

We begin our investigation of stability by showing that $p$-centrality is stable
with respect to $\delta_p$. To simplify notation, we write
\begin{equation} \label{E:theta}
\Theta_p (\mm{X}{\alpha}) = \mm{X}{\alpha,p} \,.
\end{equation}

\begin{proposition} \label{P:stab1}
Let $p \geq 1$. If $\mm{X}{\alpha}, \mm{Y}{\beta} \in \gw$ are $mm$-spaces, then
\[
\hat{r}_\infty (\mm{X}{\alpha,p} , \mm{Y}{\beta,p}) =  \var (\mathbb{C}(\mm{X}{\alpha,p}) ,\mathbb{C}(\mm{Y}{\beta,p}))
\leq \delta_p (\mm{X}{\alpha}, \mm{Y}{\beta}) \,.
\]
\end{proposition}
%------
\begin{proof}
Fix $\varepsilon > \delta_p (\mm{X}{\alpha}, \mm{Y}{\beta})$ and let
$(\phi, \psi) \in \Pi (X,Y)$ and $\mu \in \Gamma(\alpha, \beta)$
induce an $\varepsilon$-matching between $\mm{X}{\alpha}$ and
$\mm{Y}{\beta}$.  We show that $(\phi_c, \psi_c)$ is a strong $\varepsilon$-matching
between the cones on the $ft$-spaces $\mm{X}{\alpha,p}$  and $\mm{Y}{\beta,p}$.
Write
\begin{equation} \label{E:alpha}
\sigma_{p,\mm{X}{\alpha}} (x) = \left( \int_X d_X^p (x,x') \, d \alpha (x') \right)^{1/p}=
\left( \int_{X \times Y} d_X^p (x,x') \,d\mu (x',y') \right)^{1/p}
\end{equation}
and
\begin{equation} \label{E:beta}
\sigma_{p,\mm{Y}{\beta}}  (y) = \left( \int_Y d_Y^p (y, y') \, d \beta (y') \right)^{1/p}=
\left( \int_{X \times Y} d_Y^p (y, y') \,d\mu (x',y') \right)^{1/p} \,.
\end{equation}
By \eqref{E:alpha}, \eqref{E:beta} and the Minkowski inequality, 
we have that
\begin{equation} \label{E:mink}
\begin{split}
|(\sigma_{p,\mm{Y}{\beta}}  \circ \phi) (x) - \sigma_{p,\mm{X}{\alpha}} (x)|
&\leq \left( \int_{X \times Y} |d_X (x, x') -
d_Y (\phi(x), y')|^p \, d\mu (x', y') \right)^{1/p} \\
&= d_{p, \mu} (x, \phi(x)) \leq \varepsilon \,,
\end{split}
\end{equation}
$\forall x \in X$. Similarly, 
$|(\sigma_{p,\mm{X}{\alpha}}  \circ \psi) (y) - \sigma_{p,\mm{Y}{\beta}}  (y)| \leq \varepsilon$.
Thus, $(\phi, \psi)$ is a strong $\varepsilon$-pairing between 
$\mm{X}{ \sigma_{p,\mm{X}{\alpha}}}$ and $\mm{Y}{ \sigma_{p,\mm{Y}{\beta}}}$.
By Lemma \ref{L:cone}(i), $(\phi_c, \psi_c)$ is a strong $\varepsilon$-pairing
between  the cones on $\mm{X}{\alpha,p}$ and $\mm{Y}{\beta,p}$. Since $(\phi, \psi)$
and $\mu$ induce an $\varepsilon$-matching between $\mm{X}{\alpha}$ and
$\mm{Y}{\beta}$, conditions (a) and (b) in Definition \ref{D:match}
are satisfied by $(\phi_c, \psi_c)$. Hence, 
\begin{equation}
\hat{r}_\infty (\mm{X}{\alpha,p} , \mm{Y}{\beta,p}) = \var (\cone(\mm{X}{\alpha,p}),
\cone(\mm{Y}{\beta,p})) \leq \varepsilon \,.
\end{equation}
Taking infimum over $\varepsilon$-matchings, the result follows.
\end{proof}

%------

\begin{theorem}[Stability of Persistent Homology] \label{T:bottle}
Let $p \geq 1$ and $\mm{X}{\alpha}, \mm{Y}{\beta} \in \gw$.
\begin{itemize}
\item[(i)] $d_I \left((\homol^\ast_i (\mm{X}{\alpha,p}), 
\homol^\ast_i  (\mm{Y}{\beta, p}) \right) \leq
\delta_p (\mm{X}{\alpha}, \mm{Y}{\beta})$.
\item [(ii)] If $(X, d_X)$ and $(Y, d_Y)$ are triangulable, then
\[
d_B \left((PD^\ast_i (\mm{X}{\alpha,p}) ,
PD^\ast_i  (\mm{Y}{\beta, p}) \right) \leq
\delta_p (\mm{X}{\alpha}, \mm{Y}{\beta}) \,.
\]
\end{itemize}
\end{theorem}
\begin{proof}
The statements follow from Proposition \ref{P:stab1}, Corollary \ref{C:stab}
and Corollary \ref{C:diagram}.
\end{proof}

\begin{corollary} \label{C:diagrams}
Let $\alpha, \beta \in \mathcal{P} (X)$ with $(X, d_X)$ triangulable. Then,
\[
d_B \left((PD^\ast_i (\mm{X}{\alpha,p}) ,
PD^\ast_i  (\mm{X}{\beta, p}) \right)\leq
w_p (\alpha, \beta) \,,
\]
for any $p \geq 1$.
\end{corollary}
\begin{proof}
%Let $\mu \in \Gamma (\alpha, \beta)$ be a coupling that realizes
%$w_p (\alpha, \beta)$; that is, 
%\begin{equation}
%\left( \int_X d^p_X (x,x') \, d \mu (x, x') \right)^{1/p}= w_p (\alpha, \beta) \,.
%\end{equation}
%The trivial pairing $(\phi, \psi) = (I_X, I_X)$ between
%$\mm{X}{\alpha}$ and $\mm{X}{\beta}$ satisfies
%\begin{equation}
%\begin{split}
%d_{p, \mu} (\mm{X}{\alpha}, \mm{X}{\beta}) &=
%\left( \int_{X \times Y}
%|d_X(x, x_1) - d_X (x, x_2)|^p \, d\mu (x_1, x_2) \right)^{1/p} \\
%&\leq \left( \int_{X \times Y}
%d^p_X(x_1, x_2)  \, d\mu (x_1, x_2) \right)^{1/p} =
%w_p (\alpha, \beta) \,.
%\end{split}
%\end{equation}
%Thus, condition (i) in Definition \ref{D:match1} is satisfied for
%$\varepsilon = w_p (\alpha, \beta)$. Conditions (ii) and (iii)
%are trivially satisfied. Therefore, 
%\begin{equation}
%\delta_p (\Theta_p (\mm{X}{\alpha}), \Theta_p (\mm{X}{\beta}))
%\leq w_p (\alpha, \beta) \,.
%\end{equation}
This follows from Lemma \ref{L:wass} and Theorem \ref{T:bottle}.
\end{proof}

For a $mm$-space $\mm{X}{\alpha}$ and $p \geq 1$,
let $d_p^\ast (\alpha)$ be the upper $p$-Wasserstein dimension of
$\mm{X}{\alpha}$, as defined in \cite{weedbach}.

\begin{theorem}[Convergence of Persistence Diagrams]
Let $\mm{X}{\alpha}$ be a $mm$-space with $(X, d_X)$ triangulable.

\begin{itemize}
\item[(a)] (Consistency)
If $\{x_i\}_{i=1}^\infty$ are i.i.d. $X$-valued random variables with
distribution $\alpha$, then
\[
\lim_{n \to \infty} d_B \left( PD^\ast_i (\mm{X}{\alpha_n,p}) ,
PD^\ast_i  (\mm{X}{\alpha,p}) \right) =0
\]
almost surely. Here $\alpha_n$ denotes the empirical measure
$\alpha_n = \sum_{i=1}^n \delta_{x_i}/n$.
%---------
\item[(b)] (Rate of Convergence) . If $s > d_p^\ast (\alpha)$, then there is
$C>0$ such that 
\[
\mathbb{E} [d_B \left( PD^\ast_i (\mm{X}{\alpha_n,p}) ,
PD^\ast_i  (\mm{X}{\alpha,p}) \right)] \leq C \mathrm{diam} (X)\, n^{-1/s} \,,
\]
where the constant $C$ depends only on $s$ and $p$.
\end{itemize}
\end{theorem}
\begin{proof} 
(a) Varadarajan's Theorem on covergence of empirical measures guarantees that
$\alpha_n$ conveges weakly to $\alpha$ almost surely \cite{dudley02}. Since $w_p$
metrizes weak convergence of probability measures, we have that
$\lim_{n \to \infty} w_p (\alpha_n, \alpha) = 0$ almost surely. The statement
now follows from Corollary \ref{C:diagrams}.

\smallskip

\noindent (b) The inequality follows from Corollary \ref{C:diagrams}, applied
to $\alpha$ and $\alpha_n$, combined with the estimate
$\mathbb{E} [w_p (\alpha_n, \alpha)] \leq C n^{-1/s}$ by Weed and Bach
under the assumption that $\rm{diam}\, (X) \leq 1$ \cite{weedbach}.
\end{proof}

%----------------

\section{Shape of Data on Riemannian Manifolds} \label{S:manifolds}

This section specializes to data or probability measures on a fixed 
Riemannian manifold $M$. Let $(M,g)$ be a closed (compact without boundary),
connected Riemannian manifold. We denote the geodesic distance on $M$ by $d_M$,
the volume measure by $\nu$, and $V_M = \text{vol}\, (M)$. Diffusion distances associated
with the heat kernel yield a 1-parameter family of metric spaces $(M, d_t)$, $t > 0$, that may
be regarded as a multi-scale relaxation of $(M, d_M)$. We show that the persistence
diagrams $\gamma_\alpha^i (t) := PD_i^\ast (M, d_t, \alpha)$  give a continuous path in
diagram space. Moreover, the path is stable with respect to the Wasserstein distance,
thus leading to a stable, multi-scale descriptor of the shape of $\mm{M}{\alpha}$.

Let $K \colon M \times M \times (0, \infty) \to \real^+$ be the heat kernel on $M$
\cite{grigoryan}. For each $t > 0$, let $K_t \colon M \times M \to \real^+$ be
given by $K_t (x,y) = K(x,y,t)$. Define the diffusion distance of order $p \geq 1$,
at scale $t>0$, by
\begin{equation} \label{E:diff}
\begin{split}
d_t (x,x') &:= \left( \int_M |K_t (x,z) - K_t(x',z)|^p \, d\nu (z) \right)^{1/p} \\
&= \| K_t(x,\cdot)- K_t (x', \cdot)\|_p \,,
\end{split}
\end{equation}
the $\mathbb{L}_p$-distance between $K_t(x,\cdot)$ and $K_t (x', \cdot)$ with
respect to the volume measure. (We omit $p$ in the notation of the distance
because it is fixed throughout.) There is a constant $C_t  (M) > 0$, which varies
continuously with $t$, such that
\begin{equation} \label{E:lips}
|K_t (x,z) - K_t(x',z)| \leq C_t (M) d_M(x,x') \,,
\end{equation}
$\forall x, x', z \in M$ \cite[eq. (2.12)]{kasue-kumura}. It follows from \eqref{E:diff} and \eqref{E:lips} that
\begin{equation} \label{E:lipschitz}
d_t (x,x') \leq C_t (M) (V_M)^{1/p}\, d_M (x, x') \,,
\end{equation}
$\forall x,x' \in M$. In particular, this implies that the identity map
$I_M \colon (M,d_M) \to (M,d_t)$ is a homeomorphism. Thus,
$\forall s, t > 0$, the pair $(\phi, \psi) = (I_M, I_M)$ defines a pairing
between $(M, d_s)$ and $(M, d_t)$.

Given $\alpha \in \borel(M)$, we adopt the notation
$\mm{M}{\alpha} = (M, d_M, \alpha)$ and $\mm{M}{\alpha}^t =
(M, d_t, \alpha)$. The path $\gamma_\alpha \colon (0, \infty)  \to
\gw$, given by $\gamma_\alpha (t) = \mm{M}{\alpha}^t$, gives a
multi-scale relaxation of $\mm{M}{\alpha}$.

\begin{proposition} \label{P:path}
The path $\gamma_\alpha \colon (0, \infty) \to \gw$ satisfies
\[
\delta_p (\mm{M}{\alpha}^s, \mm{M}{\alpha}^{t}) \leq C_{M,p}\cdot \big|s^{-\frac{d}{2}} - t^{-\frac{d}{2}}\big|,
\]
$\forall s,t\in(0,\infty)$, where $d$ is the dimension of $M$, $C_{M,p}>0$ is a constant
that depends on $p$, on the diameter and volume of $M$, on a lower bound on its Ricci
curvature, and on $d$. In particular, $\gamma_\alpha$ is continuous with
respect to the pseudo-metric $\delta_p$ on $\gw$. 
\end{proposition}
\begin{proof}
We first establish the existence of a constant $C_{M,p}>0$ with the stated
properties such that
\begin{equation} \label{E:distortionQ}
|d_s (x,y) - d_t (x,y)| \leq C_{M,p}  \big|s^{-\frac{d}{2}} - t^{-\frac{d}{2}}\big| \,,
\end{equation}
$\forall x,y \in M$, and all $s,t\in(0,\infty)$. Write
\begin{equation}
\begin{split}
d_s (x,y) - d_{t} (x,y) &= \left( \int_M |K_s (x,z) - K_s(y,z)|^p \, 
d\nu (z) \right)^{1/p} \\
&- \left( \int_M |K_{t} (x,z) - K_{t}(y,z)|^p \, d\nu (z) \right)^{1/p} .
\end{split}
\end{equation}
By Minkowski's inequality,
\begin{equation} \label{E:mink}
\begin{split}
|d_s (x,y) - d_{t} (x,y)| &\leq \left( \int_M |K_s (x,z) - K_{t} (x,z)|^p \, 
d\nu (z) \right)^{1/p} \\
&+ \left( \int_M |K_s (y,z) - K_{t}(y,z)|^p \, d\nu (z) \right)^{1/p} .
\end{split}
\end{equation}
Now we claim that 
\begin{equation}\label{eq:claim}
|K_s (x,z) - K_{t} (x,z)| < A_M' \,\big|s^{-\frac{d}{2}} - t^{-\frac{d}{2}}\big|,\end{equation}
for some constant $A_M'>0$ that only depends on the diameter, dimension, and a lower
bound on the Ricci curvature of $M$. Notice that (\ref{eq:claim}) immediately implies
(\ref{E:distortionQ}) for $C_{M,p} = (V_M)^{1/p}\, A'_M$.

Next, we show how to conclude the proof assuming (\ref{eq:claim}).  Let $(\phi, \psi) = (I_M, I_M)$
be the trivial pairing between $(M, d_s)$ and $(M, d_{t})$ and $\mu$ be the self-coupling of $\alpha$
given by $\mu = \Delta_\ast (\alpha)$, where $\Delta \colon M \to M \times M$ is the diagonal map.
Write
\begin{equation}
\begin{split}
\sup_{x \in M} d_{p, \mu} (x, \phi (x)) &= 
\sup_{x \in M} \left( \int_{M \times M} |d_s (x,y) - d_{t} (x,y')|^p \, d\mu (y,y')\right)^{1/p}\\
&= \sup_{x \in M} \left( \int_M |d_s (x,y) - d_{t} (x,y)|^p d \alpha (y) \right)^{1/p}\\
&\leq  C_{M,p}\,\big| s^{-\frac{d}{2}} - t^{-\frac{d}{2}} \big| \,,
\end{split}
\end{equation}
where the last inequality follows from \eqref{E:distortionQ}.  Similarly,
$$\sup_{y \in M} d_{p, \mu} (y, \psi (y))  \leq C_{M,p}\,  \big|s^{-\frac{d}{2}} - t^{-\frac{d}{2}}\big|.$$
The conditions on homotopies for a $C_{M,p}  \big|s^{-\frac{d}{2}} - t^{-\frac{d}{2}}\big|$-matching
are trivially satisfied. Therefore, $\delta_p (\mm{M}{\alpha}^s, \mm{M}{\alpha}^{t}) \leq C_{M,p} 
\big|s^{-\frac{d}{2}} - t^{-\frac{d}{2}}\big|$.

Now we prove (\ref{eq:claim}). We use an argument similar to \cite[(2.12)]{kasue-kumura} and
\cite[Proposition 5.2]{dgw-spec}. First recall that \cite[Theorem 3 (iii)]{bbg} guarantees that there exists
a constant $A_M>0$ such that for all $x\in M$ and $\tau>0$,
\begin{equation}
\sum_{j\geq 1}\lambda_j e^{-\lambda_j \tau} \varphi_j^2(x)\leq A_M \,\tau^{-(\frac{d}{2}+1)} \,, 
\end{equation}
where $\lambda_j$ ($j=0,1,2,\ldots$) are the eingevalues of the Laplace-Beltrami operator
on $M$, and $\{\varphi_j\}_{j=0}^\infty$ is an orthonormal basis of $\mathbb{L}^2(M)$ consisting of
associated eigenfunctions. Then, for $x,x'\in M$ and $\tau>0$, using the spectral expansion of
$K_\tau(x,x')$ and the Cauchy-Schwartz inequality, we have
\begin{equation}
\begin{split}
\bigg|\frac{\partial K_\tau(x,x')}{\partial \tau}\bigg| &\leq \sum_{j\geq 1} \lambda_j e^{-\lambda_j \tau} 
|\varphi_j(x)|\,|\varphi_j(x')| \\
&\leq \bigg( \sum_{j\geq 1}\lambda_j e^{-\lambda_j \tau}\varphi_j^2(x) \bigg)^{1/2}
\bigg(\sum_{j\geq 1} \lambda_j e^{\lambda_j \tau}\varphi_j^2(x')\bigg)^{1/2} \,,
\end{split}
\end{equation}
which by \cite[Theorem 3 (iii)]{bbg} implies that $$\bigg|\frac{\partial K_\tau(x,x')}{\partial s}\bigg| 
\leq A_M\, \tau^{-(\frac{d}{2}+1)}.$$
Now, we may write
\begin{equation}
|K_s(x,z) - K_t(x,z)| \leq \left|\int_{s}^t \bigg|\frac{\partial K_\tau(x,x')}{\partial \tau}\bigg|\,d\tau \right|
\leq A'_M |s^{-\frac{d}{2}} - t^{-\frac{d}{2}}| \,,
\end{equation}
for some constant $A_M'>0$ only depending on the diameter, dimension, and a lower
bound on the Ricci curvature of $M$. This concludes the proof.
\end{proof}
%%%%%%%%%%%%%%%%%%%%%%%%%%%%%%%%%%%%%%%%%%%%

Before proceeding to a topological reduction of the path $\gamma_\alpha$, we 
investigate the stability of $\gamma_\alpha$ with respect to the Wasserstein distance.
Let $\chemin$ be the space of all continuous paths $(0, \infty) \to (\gw, \delta_p)$
equipped with the compact-open topology. The assignment
$\mm{M}{\alpha} \mapsto \gamma_\alpha$ defines a mapping
$R \colon \borel (M) \to \chemin$.

\begin{lemma} \label{L:tlips}
Let $t>0$ and $\alpha, \beta \in \borel (M)$. Then, 
\[
\delta_p (\mm{M}{\alpha}^t, \mm{M}{\beta}^t) \leq
C_t (M) (V_M)^{1/p} \, w_p (\alpha, \beta) \,,
\]
where the Wasserstein distance is taken with respect to geodesic
distance $d_M$.
\end{lemma}
\begin{proof}
Set $\varepsilon = C_t (M)(V_M)^{1/p} \, w_p (\alpha, \beta)$ and
let $\mu$ be a coupling between $\alpha$ and $\beta$ that realizes
$w_p (\alpha, \beta)$. By \eqref{E:lipschitz},
the pairing $(\phi, \psi) = (I_M, I_M)$ satisfies
\begin{equation}
\begin{split}
|d_t (x,x') - d_t (\phi(x), y')| &= |d_t (x,x') - d_t (x, y')| \\
&\leq d_t (x',y') \leq \,C_t (M)(V_M)^{1/p}  \, d_M (x', y') \,.
\end{split}
\end{equation}
Thus,
\begin{equation}
\begin{split}
\sup_{x \in M} \Big( \int_{M \times M} |d_t (x,x') &- d_t (\phi(x), y')|^p \,
d\mu (x', y') \Big)^{1/p} \\
&\leq C_t (M)(V_M)^{1/p}  \, w_p (\alpha, \beta) = \varepsilon\,.
\end{split}
\end{equation}
An analogous estimate is valid for $\psi$ by a similar argument. Since the
pairing $(\phi, \psi)$ is trivial, the conditions on homotopies in Definition \ref{D:match1}
are satisfied. Hence, $\mu$ and $(\phi,  \psi)$ induce an $\varepsilon$-matching
between $\mm{M}{\alpha}^t$ and $\mm{M}{\beta}^t$. This implies that
$\delta_p (\mm{M}{\alpha}^t, \mm{M}{\beta}^t) \leq \varepsilon$, 
proving the lemma.
\end{proof}

\begin{theorem} \label{T:relax}
The mapping $R \colon \borel (M) \to \chemin$ is
continuous with respect to the $p$-Wasserstein distance on $\borel (M)$
and the compact-open topology on $\chemin$ induced by $\delta_p$.
\end{theorem}
\begin{proof}
It suffices to show that given $\varepsilon > 0$ and an interval $[a,b]$,
$0 < a < b < \infty$, there is $\delta > 0$ such that
$\delta_p(\mm{M}{\alpha}^t,\mm{M}{\beta}^t) < \varepsilon$, for every $t \in [a,b]$,
provided that $w_p (\alpha, \beta) < \delta$. Set $\delta = \varepsilon /(V_M)^{1/p} A$,
where $A = \max_{t \in [a,b]} C_t (M)$. Then, Lemma \ref{L:tlips} implies that
$\delta_p(\mm{M}{\alpha}^t,\mm{M}{\beta}^t) < \varepsilon$, as desired.
\end{proof}

Let $(\diagram, d_B)$ be the space of persistence diagrams equipped
with the bottleneck distance. Given $\mm{M}{\alpha}$ and an integer $i \geq 0$, 
the path $\gamma_\alpha^i \colon (0, \infty) \to \diagram$ defined
by $\gamma_\alpha^i (t) = PD^\ast_i (\gamma_\alpha (t))$ gives a
multi-scale representation of the $i$-dimensional homology of $\mm{M}{\alpha}$.
The continuity of $\gamma_\alpha^i $ follows from Theorem \ref{T:bottle}
and Proposition \ref{P:path}.

We denote by $\chemind$ the space of all continuous path $\gamma \colon (0, \infty) 
\to \diagram$ equipped with the compact-open topology.

\begin{corollary}
Let $M$ be a closed Riemannian manifold and $i \geq 0$ an integer.
The mapping $S_i \colon \borel (M) \to \chemind$ given by
$S_i (\alpha) = \gamma_\alpha^i$ is continuous  with respect to the $p$-Wasserstein
distance on $\borel (M)$ and the compact-open topology on $\chemind$.
\end{corollary}
\begin{proof}
This is an immediate consequence of Theorems \ref{T:relax}
and \ref{T:bottle}, and the fact that $M$ is triangulable \cite{cohenetal07,desilvaetal}.
\end{proof}

%----------------

\section{Summary and Discussion} \label{S:final}

We developed a variant of persistent homology for functional data on compact topological
spaces ($ft$-spaces) and probability measures on compact metric spaces ($mm$-spaces).
Our approach highlights the topology or geometry of the signals, downplaying properties of
their domains in regions where the signals are subdued. In the functional setting, this was achieved
via a cone construction that trivializes the topology of a signal once it starts to fade away. The
construction for $mm$-spaces was essentially reduced to the functional case via centrality functions
that, among other things, have the virtue of mitigating undesired influences of data outliers on persistent
homology. We proved strong stability theorems for the proposed variant of persistent homology
with respect to metrics for which two objects are close if the corresponding signals have similar
behavior in regions where they are strong, regardless of their behavior in areas where they
fade away. Thus, in these metrics, two $ft$-spaces or $mm$-spaces may be close even
if their domains have different homotopy types. For $mm$-spaces, we also investigated
consistency and estimation of persistent homology from data.

This paper employed a 1-parameter approach to persistent homology, dealing primarily with
the theoretical aspects of the aforementioned problems. Many issues related to the computability
of the model will be treated in a forthcoming manuscript that will adopt a 2-parameter
reformulation. Some of the computational challenges relate to the fact that the
domain $X$ may be very complex or even unknown, its topological and geometric properties
having to be inferred from metric data. This suggests combining the present functional approach
with a multi-scale Vietoris-Rips type discretization of metric domains, naturally
leading to a 2-parameter formulation of persistent topology. 

To close, operations such coning functional data and truncating persistence
modules or barcodes seem to be special manifestations of a general construction
that can be performed at the level of persistence categories in a functorial manner.
This is another angle of the problems treated in this paper that will
deferred to a future investigation.

\bibliographystyle{plain}
\bibliography{mmspaces.bib}

\end{document}